\documentclass[12pt,amscd]{amsart}
\usepackage[all]{xy}
\usepackage{graphicx}
\usepackage{amsmath,amsxtra,amssymb,latexsym, amscd,amsthm}
\usepackage{indentfirst}
\usepackage[mathscr]{eucal}

\newtheorem{thm}{Theorem}[section]
\newtheorem{cor}[thm]{Corollary}
\newtheorem{lem}[thm]{Lemma}

\theoremstyle{definition}
\newtheorem{defn}[thm]{Definition}
\newtheorem{exm}[thm]{Example}
\newtheorem{rem}[thm]{Remark}

\numberwithin{equation}{section}

\DeclareMathOperator{\depth}{depth}
\DeclareMathOperator{\ass}{Ass}

\DeclareMathOperator{\supp}{supp}

\DeclareMathOperator{\dstab}{dstab}
\DeclareMathOperator{\astab}{astab}
\DeclareMathOperator{\conv}{conv}
\DeclareMathOperator{\height}{ht}
\DeclareMathOperator{\link}{lk}

\def\Nset{\mathbb {N}}
\def\Zset{\mathbb {Z}}

\def\Rset{\mathbb {R}}
\def\Ga {CS_{\albf}}
\def\Da {\Delta_{\albf}}

\def\albf {{\boldsymbol{\alpha}}}
\def\bebf {{\boldsymbol{\beta}}}

\def\gabf {{\boldsymbol{\gamma}}}

\def\zbf {\mathbf 0}
\def\ebf {\mathbf e}
\def\xbf {\mathbf x}
\def\ybf {\mathbf y}
\def\abf {\mathbf a}

\def\mfr {\mathfrak m}
\def\nfr {\mathfrak n}

\def\afr {\mathfrak a}

\def\Htil {\widetilde{H}}

\begin{document}

\title[Stability of Depths and Cohen-Macaulayness of Integral Closures] {Stability of Depth and Cohen-Macaulayness of Integral Closures of  Powers of Monomial Ideals}
\author{Le Tuan Hoa and Tran Nam Trung}
\address{Institute of Mathematics, VAST, 18 Hoang Quoc Viet, 10307 Hanoi, Viet Nam}
\email{lthoa@math.ac.vn, tntrung@math.ac.vn}
\date{}
\dedicatory{}
\commby{}
\begin{abstract} Let $I$ be  a monomial ideal $I$ in a polynomial ring $R = k[x_1,...,x_r]$. In this paper we give an upper bound on  $\overline{\dstab} (I)$  in terms of $r$ and the maximal generating degree $d(I)$ of $I$ such that  $\depth R/\overline{I^n}$ is  constant for all $n\geqslant \overline{\dstab}(I)$.  As an application, we classify the class of monomial ideals $I$ such that $\overline{I^n}$ is Cohen-Macaulay for  some integer $n\gg 0$.
\end{abstract}

\maketitle
\section*{Introduction}

Let $R = k[x_1,\ldots, x_r]$ be a polynomial ring over a field $k$ and $\afr$ a homogeneous ideal in $R$.  It was shown by Brodmann \cite{B}  that $\depth R/\afr^n$ is  constant for $n\gg 0$. The smallest integer $m>0$ such that $\depth R/\afr^n = \depth R/\afr^m$ for all $n\geqslant m$ is called the index of depth stability and is denoted by $\dstab(\afr)$. Since the behavior of depth function $\depth R/\afr^n$ is quite mysterious (see \cite{HH1, HHTT}), it is of great interest to bound $\dstab(\afr)$ in terms of $r$ and $\afr$. However, until now this problem is only solved for a few classes of monomial ideals (see, e.g.,  \cite{HH1, HQ, Tr2}). The bound obtained in \cite{Tr2} for ideals generated by square-free monomials of degree two is rather small and optimal. However, this problem is still open for a general square-free monomial ideal.

In this direction, it is also of interest to consider similar problems for other powers of $\afr$. In \cite{HKTT} together with Kimura and Terai we were able to solve the problem of bounding the index of depth stability for symbolic powers of square-free monomial ideals. In this paper we are interested in bounding the index of depth stability $\overline{\dstab}(\afr)$ for integral closures,  which is defined as the smallest integer $m>0$ such that $\depth R/\overline{\afr^n} = \depth R/\overline{\afr^m}$ for all $n\geqslant m$.  Like in the case of ordinary powers, $\overline{\dstab}(\afr)$ is well-defined. We  only consider the problem for monomial ideals $I$. In this context one can use geometry and  convex analysis to describe the integral closures of $I^n$ (see Definition \ref{NP} and some properties after it). Then one can use Takayama's formula (see Lemma \ref{Takay}) to compute the local cohomology modules of $R/\overline{I^n}$. This approach was successfully applied in several papers (see, e.g.,  \cite{HKTT, HT, Tr1}). In particular, one can show that in the class of monomial ideals the behavior of the function $\depth R/\overline{I^n}$ is much better than that of $\depth R/I^n$: it is ``quasi-decreasing" (see Lemma \ref{DepthInc}) while the  function $\depth R/I^n$ can be any convergent non-negative numerical function  (see \cite{HHTT}). Our main result is Theorem \ref{DStab}, where we can give an upper bound on $\overline{\dstab}(I)$ in terms of $r$ and the maximal generating degree $d(I)$ of $I$ for any monomial ideal $I$. Although our bound is very big, an example shows that an upper  bound must depend on $d(I)$, and in the worst case must be an exponential function of $r$.

In order to bound $\overline{\dstab}(I)$ we have to study the  index of stability for the  associated primes on $R/\overline{I^n}$. This in some sense  corresponds the zero depth case and was firstly done in \cite{Tr1}. In this paper we can improve the main result of \cite{Tr1} by giving an essentially better bound, see Theorem \ref{AssTh}.

As an application we classify all monomial ideals such that $R/\overline{I^n}$ is a Cohen-Macaulay ring for all $n\geqslant 1$ (or for some fixed $n= n_0\gg 0$). It turns out  that only equimultiple ideals have this property, see Theorem \ref{CM}. In the case of square-free monomial ideals, we can then derive a criterion for the Cohen-Macaulayness of $R/\overline{I^n}$  for some fixed $n\geqslant 3$, see Theorem \ref{CMDelta}. This criterion is exactly the one for the Cohen-Macaulayness of $R/{I^n}$ given in \cite[Theorem 1.2]{TT}.

The paper is organized as follows. In Section \ref{Ass} we study the stability of associated primes and give an upper bound on $\overline{\astab}(I)$ of a monomial ideal. In Section \ref{Depth} we prove the main Theorem \ref{DStab}. The study of Cohen-Macaulay property of $R/\overline{I^n}$ is done in  the last section.

\section{Stability of associated primes} \label{Ass}

Let $R := k[x_1, \ldots , x_r]$ be a polynomial ring over a field $k$ with the maximal homogeneous ideal $\mfr = (x_1, \ldots , x_r)$.  
Throughout this paper,  let $I$ be  a proper monomial ideal in $R$.  Let $\Nset,\ \Rset,\ \Rset_+$ be the set of non-negative integers, real numbers and non-negative real numbers, respectively.  For a vector $\albf =(\alpha_1,\ldots,\alpha_r)\in\Nset^r$, we denote by $\xbf^{\albf} = x_1^{\alpha_1}\cdots x_r^{\alpha_r}$.  

The {\it integral closure} of an arbitrary ideal $\afr$ of $R$ is the set of elements $x$ in $R$ that satisfy an integral relation
$$x^n + a_1x^{n-1} + \cdots + a_{n-1}x + a_n = 0,$$
where $a_i \in \afr^i$ for $i = 1,\ldots, n$. This is an ideal and is denoted by $\overline \afr$. 
The integral closure of  a monomial ideal $I$  is a monomial ideal as well. We can geometrically describe $\overline{I}$ by using its Newton polyhedron.

\begin{defn} \label{NP} Let $I$ be a monomial ideal of $R$. We define
\begin{enumerate}
\item  For a subset $A \subseteq R$, the exponent set of $A$ is $E(A) := \{\albf \mid \xbf^{\albf}\in A\} \subseteq  \Nset^r$.
\smallskip
\item The Newton polyhedron of $I$ is $NP(I) := \conv\{E(I)\}$, the convex hull of the exponent set of $I$ in
the space $\Rset^r$.
\end{enumerate}
\end{defn}

The following results are well-known (see \cite{RRV}):
\begin{equation}\label{EN1}
E(\overline I) = NP(I) \cap \Nset^r = \{\albf \in \Nset^r \mid \xbf^{n\albf} \in I^n \text{ for some } n\geqslant 1\}.
\end{equation}
\begin{equation}\label{EN2}
NP(I^n) =nNP(I)= n\conv\{E(I)\} +\Rset_{+}^r \ \text{for all}\ n\geqslant 1.
\end{equation}

Let $G(I)$ denote the minimal generating system of monomials of $I$ and
$$d(I) := \max\{\alpha_1 + \cdots + \alpha_r |\ \xbf^\albf \in G(I) \},$$
the maximal generating degree  of $I$. Let $\ebf_1,...,\ebf_r$ be the canonical basis of $\Rset^r$.
The first part of the following result is \cite[Lemma 6]{Tr1}. It gives  more precise information on the coefficients of defining equations of supporting hyperplanes of $NP(I)$.

\begin{lem}\label{NPH} The Newton polyhedron $NP(I)$ is the set of solutions of a system of inequalities of the form
$$\{\xbf\in\Rset^r \mid \left<\abf_j, \xbf\right> \geqslant b_j,\  j=1,\ldots,q\},$$
such that each hyperplane with the equation $\left<\abf_j,\xbf\right> = b_j$ defines a facet of $NP(I)$, which contains $s_j$ affinely
independent points of $E(G(I))$ and is parallel to $r - s_j$ vectors of the canonical basis. Furthermore, we can
choose $\zbf \ne \abf_j \in \Nset^r, b_j \in \Nset$ for all $j = 1, . . . , q$; and if we write $\abf_j=(a_{j1},\ldots,a_{jr})$, then
$$a_{ji} \leqslant s_jd(I)^{s_j-1} \ \text{ for all } i=1,\ldots, r,$$
where $s_j$ is the number of non-zero coordinates of $\abf_j$.
\end{lem}

\begin{proof} The first part of the lemma is \cite[Lemma 6]{Tr1}. Moreover, it also claims that $\abf_j \in \Rset_+^r$ and $b\in \Rset_+$.  For the second part,  let $H$ be a hyperplane which defines a facet of $NP(I)$.  W.l.o.g, we may assume that $H$ is defined by $s$ affinely independent points $\albf_1,...,\albf_s \in E(G(I))$ and is parallel to $r-s$ vectors  $\ebf_{s+1},\ldots,\ebf_r$.  Then the defining equation of $H$ can be written as
$$
\left | \begin{array}{cccc}
x_1& \cdots & x_s & 1\\
\alpha_{11}& \cdots & \alpha_{1s} & 1\\
\vdots & \vdots & \vdots & \vdots\\
\alpha_{s1}& \cdots & \alpha_{ss} & 1
\end{array}
\right | = 0.
$$
Expanding this determinant in the first row, we get: $a'_1x_1+\cdots+ a'_sx_s = b'$, where $a'_i$ are the $(1,i)$-cofactor for $i=1,\ldots,s$ and $b'$ is the $(1,s+1)$-cofactor of this determinant. Clearly, $a'_1,\ldots, a'_s,b'\in\Zset$.  Note that we may take $a_i = |a'_i|$ and $b= |b'|$. Expanding the determinant
$$
a'_1=\left | \begin{array}{cccc}
\alpha_{12}& \cdots & \alpha_{1s} & 1\\
\vdots & \vdots & \vdots & \vdots\\
\alpha_{s2}& \cdots & \alpha_{ss} & 1
\end{array}
\right | ,
$$
 in the last column, we get  
$$
(-1)^{s+1} a'_1=\left | \begin{array}{ccc}
\alpha_{22}& \cdots & \alpha_{2s}\\
\alpha_{32}& \cdots & \alpha_{3s}\\
\vdots & \vdots & \vdots\\
\alpha_{s2}& \cdots & \alpha_{ss}
\end{array}
\right |
-
\left | \begin{array}{ccc}
\alpha_{12}& \cdots & \alpha_{1s}\\
\alpha_{32}& \cdots & \alpha_{3s}\\
\vdots & \vdots & \vdots\\
\alpha_{s2}& \cdots & \alpha_{ss}
\end{array}
\right |
+\cdots+(-1)^{s-1}
\left | \begin{array}{ccc}
\alpha_{12}& \cdots & \alpha_{1s}\\
\alpha_{22}& \cdots & \alpha_{2s}\\
\vdots & \vdots & \vdots\\
\alpha_{s-1,2}& \cdots & \alpha_{s-1,s}
\end{array}
\right |.
$$
Let $\det(c_{ij})$ be a determinant in the above sum. By Hadamard's inequality, we have
$$(\det (c_{ij}))^2 \leqslant \Pi_{i=1}^{s-1}(\Sigma_{j=1}^{s-1}|c_{ij}|^2) \leqslant   \Pi_{i=1}^{s-1}(\Sigma_{j=1}^{s-1}|c_{ij}|)^2 \leqslant d(I)^{2(s-1)}.$$
Hence  $a_1 = |a'_1| \leqslant sd(I)^{s-1}$. Similarly, $a_i \leqslant sd(I)^{s-1}$ for $i=2,\ldots,r$, as required.
\end{proof}

 The following lemma is a crucial result in the study of the stability of  $\ass(R/\overline{I^n})$.
 
\begin{lem} \label{AssL2} Let $I$ be a monomial ideal in $R$ with $r > 2$. If $\mfr \in \ass R/\overline{I^s}$ for some $s\geqslant 1$, then
$$\mfr \in \ass R/\overline{I^n} \text{ for all } n\geqslant (r-1)r d(I)^{r-2}.$$
\end{lem}

\begin{proof} Let $m := (r-1)rd(I)^{r-2}$. Since the sequence $\{\ass R/\overline{I^n}\}_{n\geqslant 1}$ is increasing by  \cite[Proposition 16.3]{HIO},  it suffices to show that $\mfr \in \ass R/\overline{I^m}$.

As $\mfr \in \ass R/\overline{I^s}$, by \cite[Lemma 13]{Tr1}, there is a supporting hyperplane of $NP(I)$, say $H$,  of the form $\left<\abf, \xbf \right> = b$ such that all coordinates of $\abf$ are positive. By Lemma \ref{NPH}, this hyperplane passes through $r$ affinely independent points of $E(G(I))$, say $\albf_1, \ldots ,\albf_r$. Let $J :=(\xbf^{\albf_1},\ldots,\xbf^{\albf_r})$. Clearly, $H$ is still a supporting plane of $NP(J)$. Again by Lemma \ref{NPH}, the Newton polyhedron $NP(J)$ can be represented by a system of inequalities
$$\{\xbf \in\Rset^r \mid \left<\abf_j,\xbf \right> = b_j, j=1,\ldots, q\},$$
where $\zbf \ne \abf_j\in\Nset^r$ and $b_j \in\Nset$. Let  $H_j =\{\xbf \in\Rset ^r\mid \left<\abf_j, \xbf\right> = b_j\}$ for $j=1,\ldots,q$. We may assume that  $q$ is minimal and $H_q = H$. Since $J$ is generated by exactly $r$ monomials and $q$ is taken to be minimal, by Lemma \ref{NPH}, each hyperplane $H_j$,  where $j\leqslant q-1$, must be parallel to at least one of the vectors $\ebf_1,\ldots,\ebf_r$.  Hence, by the second statement of Lemma \ref{NPH}, we may assume that
\begin{equation}\label{ET1}
a_{ji} \leqslant (r-1)d(J)^{r-2} \ \text{ for all } j \leqslant q-1\text{ and } i \leqslant r.
\end{equation}
Consider the barycenter $\albf := \frac{1}{r} (\albf_1 + \cdots + \albf_r)$ of the simplex $[\albf_1, \ldots ,\albf_r]$. Then $\albf$ is a relative interior point of the facet $H_q\cap NP(J)$ of $NP(J)$. Therefore, $\albf$ does not lie in $H_j$ for all $j =1,\ldots,q-1$, and so
\begin{equation} \label{ET1a} \left<\abf_j,\albf \right> > b_j \text{ for all } j \leqslant q-1. \end{equation}
Next, we may assume that $a_{qr} =\min\{a_{q1},\ldots,a_{qr}\} > 0$. Let $\bebf:= m\albf -\ebf_r$. Then $\bebf = (r-1)d(I)^{r-2} (\albf_1 + \cdots + \albf_r) - \ebf_r   \in \Zset^r$.  Since $\albf_1,...,\albf_r\in H_q$ are affinely independent and $ a_{q1}, ..., a_{qr}>0$, there exists $j\leqslant r$ such that $\albf_{jr} >0$, whence $\albf_{jr} \geqslant 1$. Hence $\bebf \in \Nset^r$. Moreover,
 $$\left<\abf_q,\bebf \right> =  m\left<\abf_q,\albf \right> - \left<\abf_q,\ebf_r\right> = m b_q-a_{qr} < m b_q.$$ 
 Therefore $\bebf \notin NP(J^m)$ and also $\bebf \notin NP(I^m)$ (recall that $H = H_q$). 

On the other hand, we claim that
\begin{equation} \label{ET1b} \bebf +\ebf_i \in NP(J^m) \ \text{ for all } i=1,\ldots, r. \end{equation}
Indeed,  for $i=r$,  $\bebf + \ebf_r = m\albf \in mNP(J) = NP(J^m)$. For $i \leqslant r-1$, we have
$$\begin{array}{ll} \left<\abf_q,\bebf +\ebf_i\right> & = \left<\abf_q, m\albf -\ebf_r+\ebf _i\right> \\
& = m\left<\abf_q,\albf\right> -\left<\abf_q,\ebf_r\right>+\left<\abf_q,\ebf_i\right> \\
& = mb_q-a_{qr}+a_{qi} \\
& \geqslant mb_q \ \  (\text{since } a_{qr} =\min\{a_{q1},\ldots,a_{qr}\} ) .\end{array}$$
Let $j \leqslant q-1$.  Since $r\albf = \albf_1+\cdots+\albf_r\in NP(J^r)\cap \Nset^r$,  by (\ref{ET1a}), we have $\left<\abf_j,r\albf \right> > rb_j$, which implies  $\left<\abf_j,r\albf\right> \geqslant rb_j +1$. Hence
$$\begin{array}{ll}
\left<\abf_j,\bebf +\ebf_i\right> &=\left<\abf_j,m\albf -\ebf_r+\ebf_i\right> \\
& = (r-1)d(I)^{r-2}\left<\abf_j,r\albf \right> -\left<\abf_j,\ebf_r\right>+\left<\abf_j,\ebf_i\right>\\
&\geqslant (r-1)d(I)^{r-2} (rb_j+1) - a_{jr}+a_{ji} \\
& =  mb_j +((r-1)d(I)^{r-2}  -a_{jr}) + a_{ji} \\
& \geqslant mb_j \ \ \text{(by (\ref{ET1}))} . \end{array}$$
This completes the proof of (\ref{ET1b}). 

Since  $NP(J^m)  = m NP(J) \subseteq m NP(I) = NP(I^m)$,   $\bebf+\ebf_i\in NP(I^m)$, whence  $\xbf^{\bebf}x_i \in \overline{I^m}$.  As shown above, $\bebf \not \in NP(I^m)$. Therefore, $\mfr \in \ass R/\overline{I^m}$, as required.
\end{proof}

A main tool in the study of the set of associated primes and the depth of rings is using local cohomology modules.  In the setting of monomial ideals,  one often uses a generalized version of a Hochster's formula given by Takayama in \cite{Ta}. Let us recall this formula here.

Since $R/I$ is an $\Nset^r$-graded algebra, $H_{\mfr}^i(R/I)$ is an $\Zset^r$-graded module over $R$.  For every degree $\albf\in\Zset^r$ we denote by $H_{\mfr}^i(R/I)_{\albf}$ the $\albf$-component of $H_{\mfr}^i(R/I)$.

Let $\Delta(I)$ denote the simplicial complex corresponding to the Stanley-Reisner ideal $\sqrt I$, i.e.
$$\Delta (I) =\{ \{i_1,...,i_s\} \subseteq [r]|\ x_{i_1}\cdots x_{i_s} \not\in \sqrt I\} ,$$
where $[r]$ denotes the set $\{1,2,...,r\}$.  For every $\albf = (\alpha_1,\ldots,\alpha_r) \in \Zset^r$, we define its co-support  to be the set $\Ga := \{i \ | \ \alpha_i < 0\}$.  For a subset $F$ of $[r]$, let $R_F := R[x_i^{-1} \ | \ i \in F]$.  
Set
\begin{equation} \label{EQ01}  \Da(I) = \{ F \subseteq [r]\setminus \Ga|\  \xbf^\albf \notin IR_{F\cup \Ga} \}. \end{equation}
We set $\widetilde{H}_i(\emptyset ;k) = 0$ for all $i$,   $\widetilde{H}_i( \{\emptyset \} ;k) = 0$ for all $i\neq -1$, and $\widetilde{H}_{-1} (\{ \emptyset\} ;k) = k$. Thanks to \cite[Lemma  1.1]{GH} we may formulate Takayama's formula as follows.

 \begin{lem}\label{Takay} {\rm (\cite[Theorem 2.2]{Ta})} $\dim_k H_{\mfr}^i(R/I)_{\albf} = \dim_k \widetilde{H}_{i-| \Ga|-1}(\Da (I);k).$
\end{lem}

As an immediate consequence of this result is the  following``quasi-decreasing" property of the depth function $\depth  R/\overline{I^n}$. We don't know if this property holds for an arbitrary homogeneous  ideal.

\begin{lem}\label{DepthInc} For any monomial ideal $I$ of $R$, we have
\begin{enumerate}
\item $\depth R/\overline{I^m}\geqslant \depth R/\overline{I^{mn}} \text{ for all } m, n\geqslant 1$.
\medskip
\item $\lim_{n\rightarrow \infty}\depth R/\overline{I^n} =\dim R-\ell(I),$ where $\ell (I)$ denotes the analytic spread of $I$.
\end{enumerate}
\end{lem}

\begin{proof} 1) Replacing $I^m$ by $J$, it suffices to prove the statement for $m=1$. Let $t := \depth R/ \overline{I}$. Then we must have $H^t_{\mfr}(R/\overline{I})_{\albf} \ne 0$ for some $\albf \in \Zset^r$.  By Lemma \ref{Takay}, 
\begin{equation}\label{EI1}
\dim_k \Htil_{t-|\Ga|-1}(\Da(\overline{I});k) = \dim_k H^t_{\mfr}(R/\overline{I})_{\albf} \ne 0.
\end{equation}

For $n\geqslant 1$, we have $CS_{n\albf}= \Ga$ and
$$\Da(\overline{I}) = \{F\in\Delta \mid \xbf^{\albf} \notin \overline{I R_{F\cup \Ga}}\} = \{F\in\Delta \mid \xbf^{n\albf} \notin \overline{(I R_{F\cup \Ga})^n}\} = \Delta_{n\albf}(\overline {I^n}).$$
The middle equality follows from (\ref{EN1}). Together with Equation (\ref{EI1}) and Lemma \ref{Takay}, this fact implies that
$$\dim_k H^t_{\mfr}(R/\overline{I^n})_{n\albf} = \dim_k \Htil_{t -|CS_{n\albf}|-1}(\Delta_{n\albf}(\overline{I^n});k) = \dim_k \Htil_{t -|CS_{\albf}|-1}(\Delta_{\albf}(\overline{I});k) \ne 0.$$
This means $\depth R/\overline{I^n} \leqslant t$.

\medskip
2) Let $J:=\overline{I^{r-1}}$. By \cite[Theorem 7.29]{Vs}, $J$ is  torsion-free. Thus, by \cite[Proposition $3.3$]{EH}, we have
$$\lim_{m\rightarrow \infty}\depth R/J^m =\dim R-\ell(J).$$

For each $m\geqslant 1$, by  \cite[Corollary 7.60]{Vs}, we have $J^m = \overline{I^{m(r-1)}}$. Hence 
$$\lim_{n\rightarrow \infty}\depth R/\overline{I^n} = \lim_{m\rightarrow \infty}\depth R/J^m = \dim R-\ell(J).$$
Note that $\ell(J) =\ell(\overline{I^{r-1}}) = \ell(I^{r-1}) =\ell(I)$, so the desired equality follows.
\end{proof}

Let $F$ be a subset of $[r]$. Put $R[F] = k[x_i\mid i\notin F]$ and denote by $I[F]$ the ideal of $R[F]$ obtained from $I$ by setting $x_i = 1$ for all $i\in F$. Then 
$I[F] R = IR_F \cap R$.
If $F=\{i\}$ for some $i\in [r]$,  then we write $R[i]$ and $I[i]$ instead of $R[\{i\}]$ and $I[\{i\}]$ respectively.

\begin{rem} \label{Restriction} Let $I$ be a monomial ideal in $R$. Then,
\begin{enumerate}
\item Using (\ref{EN1}) it is easy to see that $\overline{I^n}[F] = \overline{I[F]^n}$ for any $n\geqslant 1$ (cf. \cite[Lemma 4.6]{HT}).
\item If $I$ is Cohen-Macaulay, then $I[F]$ is Cohen-Macaulay.
\end{enumerate}
\end{rem}

We can now give an improvement  of the main result, Theorem 16, in  \cite{Tr1}.

\begin{thm} \label{AssTh} Let $I$ be a monomial ideal of $R$ and
\begin{equation*}n_0(I):=
\begin{cases}
1 & \text{ if } \ell(I) \leqslant 2,\\
\ell(I)(\ell(I)-1)d(I)^{\ell(I)-2} & \text{ if } \ell(I) > 2.
\end{cases}
\end{equation*}
Then, $\ass R/ \overline{I^n} = \ass R/\overline{I^{n_0(I)}}$ for all $n\geqslant n_0(I)$.
\end{thm}

\begin{proof} Fix an index  $i\leqslant r$. It is well known that the analytic spread of  $I$ is equal to the minimal  number of generators of a minimal reduction of  $I$. Since $I[i]$ is obtained from $I$ by setting $x_i =1$, this implies that $\ell(I) \geqslant \ell(I[i])$  and $d(I[i]) \leqslant d(I)$. Hence  $n_0(I[i]) \leqslant n_0(I)$. Using this remark, Remark \ref{Restriction}(1) and Lemma \ref{DepthInc},  we can  prove the theorem by induction on $r$. The proof is similar to that of \cite[Theorem 16]{Tr1},  so we omit  details here.
\end{proof}

\noindent {\it Remark}. Set 
$$ \overline{\astab}(I) = \min\{m |\  \ass R/ \overline{I^n} = \ass R/\overline{I^m} \ \text{for all } n\geqslant m \}.$$
It can be called the  index of stability for the  associated primes of $R/\overline{I^n}$. An example given in \cite[Proposition 17]{Tr1} shows that an upper bound on $ \overline{\astab}(I)$ must be of the order $d(I)^{r-2}$, provided that $r$ is fixed. The coefficient of $d(I)^{r-2}$ in the upper bound given in \cite[Theorem 16]{Tr1} is $r 2^{r-1}$.

\section{Stability of Depth} \label{Depth}

In this section we study the stability index of the depth function $\depth R/\overline{I^n}$. It is clear that a simplicial complex $\Delta$ is defined by the set of its maximal faces, say $F_1,...,F_s$. In this case we write $\Delta = \left< F_1,...,F_s\right>$.  Keeping the notations in Lemma \ref{NPH}, we set $\supp(\abf_j) := \{i \mid a_{ji} \ne 0\}$.  
 We can describe $\Da(\overline{I^n})$  as follows.

\begin{lem} \label{FNPn} For any $\albf \in \Nset^r$ and $n\geqslant 1$, we have
$$\Da(\overline{I^n}) =\left< [r]\setminus \supp(\abf_j) \mid j\in \{1,\ldots,q\} \text{ and }  \left<\abf_j,\albf \right> <nb_j  \right>.$$
\end{lem}

\begin{proof} Let $F\in \Da(\overline{I^n})$. We may assume that $F =\{s+1,\ldots,r\}$ for some $0\leqslant s\leqslant r$.  By Lemma \ref{NPH} and (\ref{EN2}), we can deduce  that $NP(I^n)$ is the set of solutions of the system
$$\{\xbf \in \Rset^r \mid \left<\abf_j, \xbf \right> \geqslant nb_j, \ j=1,\ldots,q\}.$$
Since $CS_{\albf} = \emptyset $, $\xbf^{\albf} \notin \overline{I^n}R_{F  }$ if and only if $\xbf^{\albf} \xbf^{\gabf} \notin \overline{I^n}$ for any monomial $\xbf^{\gabf}\in k[x_{s+1},\ldots,x_r]$. Taking $\xbf^{\gabf} = x_{s+1}^m\cdots x_r^m$,  where 
$$m> \max\{n b_j -  \left<\abf_j,\albf\right> \mid \ 1\le j\leqslant q\}, $$ 
is fixed, it implies that  there is $1\le p\leqslant q$ such that
$$\left<\abf_p,\albf+ m( \ebf_{s+1} + \ebf_r) \right> < nb_p .$$
Assume that there is $i\geqslant s+1$ such that $a_{pi} >0$. Then 
$$ \left<\abf_p,\albf+ m( \ebf_{s+1} + \ebf_r) \right> = \left<\abf_p,\albf\right> + m(a_{p(s+1)} + \cdots a_{pr}) \geqslant \left<\abf_p,\albf\right> + m >  nb_p ,$$ 
a contradiction. Hence $a_{p(s+1)} = \cdots = a_{pr} = 0$, whence $F \subseteq [r] \setminus \supp(\abf_p)$ . Then $\left<\abf_p,\albf \right> =  \left<\abf_p,\albf+ m( \ebf_{s+1} + \ebf_r) \right> < nb_p$. 

Conversely, assume that  there is $j\leqslant q$ such that $F \subseteq [r] \setminus \supp(\abf_j)$, i.e. $a_{j(s+1)} = \cdots = a_{jr} = 0$, and 
 $\left<\abf_j,\albf \right> <nb_j $. Then  for all monomials $\xbf^{\gabf}\in k[x_{s+1},\ldots,x_r]$, we have
 $ \left<\abf_j,\albf + \gabf  \right>   = \left<\abf_j,\albf \right>  < nb_j.$
 By (\ref{EN1}) and Lemma \ref{NPH}, this implies $\xbf^{\albf} \xbf^{\gabf} \notin \overline{I^n}$. From (\ref{EQ01}), we get that $F\in \Da(\overline{I^n})$. This completes the proof of the lemma.
 \end{proof}
 
 The following lemma is the main step in the proof of Theorem \ref{DStab}
 
 \begin{lem}\label{DStabL} Let $m\geqslant 1$ and $t:=\depth R/\overline{I^m}$. Assume that $H_{\mfr}^t(R/\overline{I^m})_{\bebf} \ne 0$ for some $\bebf\in \Nset^r$. If $r\geqslant 3$, then 
$$\depth R/\overline{I^n} \leqslant t \ \text{ for all } n\geqslant r(r^2 - 1)r^{r/2}(r-1)^rd(I)^{(r-2)(r+1)}.$$
\end{lem}
\begin{proof} For simplicity, set $n^* : = r(r^2 - 1)r^{r/2}(r-1)^rd(I)^{(r-2)(r+1)}$. We keep the notations in Lemma \ref{NPH}.
\medskip

Assume that $\supp(\abf_j) = [r]$ for some $1\leqslant j \leqslant q$.  By \cite[Lemma 14]{Tr1} we have $\mfr \in \ass R/\overline{I^n}$ for all $n\gg 0$. By Lemma \ref{AssL2},  it yields $\mfr \in \ass R/\overline{I^n}$ for all $n\geqslant n^*$. Thus, $\depth R/\overline{I^n} = 0$ for $n\geqslant n^*$, and the lemma holds in this case.
\medskip

We now assume that $\supp(\abf_j) \ne [r]$ for all $j=1,\ldots,q$, i.e. the number of non-zero coordinates of $\abf_j$ is strictly less than $r$. By Lemma \ref{NPH}, we have
\begin{equation}\label{EDStabL1}
a_{ji} \leqslant (r-1)d(I)^{r-2} \ \text{ for } i=1,\ldots,r.
\end{equation}
Assume that
$$\left<\abf_j, \bebf \right> < b_j \text{ for } j =1,\ldots, p, $$
and
$$\left<\abf_j, \bebf \right> \geqslant b_j  \text{ for } j  =p+1,\ldots, q,$$
for some $0\leqslant p\leqslant q$. Then, by Lemma $\ref{FNPn}$,
$$\Delta_{\bebf}(\overline{I^m}) = \left<[r]\setminus \supp(\abf_j) \mid j=1,\ldots,p\right>.$$
By Lemma \ref{Takay}, we have
$$\dim_k \Htil_{t-1}(\Delta_{\bebf}(\overline{I^m});k) = \dim_k H_{\mfr}^t(R/\overline{I^m})_{\bebf} \ne 0.$$
Hence $\Delta_{\bebf}(\overline{I^m})$ is not acyclic. In particular,  $p\geqslant 1$.
For each $n \geqslant 1$, put
$$\Gamma(\overline{I^n}) := \{\albf \in \Nset^r \mid \Delta_{\albf}(\overline{I^n}) = \Delta_{\bebf}(\overline{I^m})\},$$
and 
\begin{equation}\label{EDStabL2}
C_n:=\{\xbf \in\Rset^r\mid \left<\abf_j, \xbf \right> < nb_j,  \left<\abf_l, \xbf \right> \geqslant  nb_l \text{ for } j \leqslant p; \  p+1\le l \leqslant q\} \subseteq \Rset_+^r.
\end{equation}
It is clear that $C_n = nC_1$.  By Lemma \ref{FNPn},   $C_n \cap \Nset^r \subseteq \Gamma(\overline{I^n})$.

Assume that  $C_n \cap \Nset^r  \neq \emptyset$.  Then  for any $\albf\in C_n \cap \Nset^r$,  by Lemma \ref{Takay}, we have
$$ \dim_k H_{\mfr}^t(R/\overline{I^n}))_{\albf}=\dim_k \Htil_{t-1}(\Delta_{\albf}(\overline{I^n});k)=\dim_k \Htil_{t-1}(\Delta_{\bebf}(\overline{I^m});k)\ne 0,$$
whence  $\depth R/ \overline{I^n} \leqslant t$. 

Thus, in order to complete the proof of the lemma, it remains to show that $C_n \cap \Nset^r \ne \emptyset$ for any $n\geqslant n^*$. Fix such an integer $n$.
\medskip

Since $\bebf \in C_m = mC_1$, $C_1\ne\emptyset$. First, we  prove that $C_1$ is bounded in $\Rset^r$. Assume  that $a_{ji} = 0$ for some $1\leqslant i \leqslant r$ and for all $j=1,\ldots,p$. Then, for any $s \gg 0$, by Formula (\ref{EDStabL2}) we get that $\bebf +s\ebf_i \in C_m$, which implies $\Delta_{\bebf+s\ebf_i}(\overline{I^m})=\Delta_{\bebf}(\overline{I^m})$.  Again by Lemma \ref{Takay},  we have
$$\dim_k H_{\mfr}^t(R/\overline{I^m})_{\bebf + s\ebf_i} = \dim_k \Htil_{t-1}(\Delta_{\bebf+s\ebf_i}(\overline{I^m});k) =  \dim_k \Htil_{t-1}(\Delta_{\bebf}(\overline{I^m});k)\ne 0.$$
This contradicts the Artiness of  $H_{\mfr}^t (R/\overline{I^m})$. Hence, for each $i\leqslant r$, there is $j_i\leqslant p$ such that $a_{j_ii} \geqslant 1$.
Let $\ybf=(y_1,\ldots,y_r)$ be an arbitrary point of $C_1 \subseteq \Rset_+^r$ .  Then for each $i\leqslant r$, we have
$y_i \leqslant a_{j_i i}y_i \leqslant \left<\abf_{j_i},\ybf \right>  < b_{j_i}$.  This implies that  $C_1$ is bounded and so is $C_n$.
\medskip

Let $\overline C_n$  be the closure of $C_n$ in $\Rset^r$ with respect to  the usual Euclidean topology. Then $\overline{C_n}$ is bounded as well. Moreover,
\begin{equation}\label{EDStabL3}
\overline C_n=\{\xbf \in\Rset^r\mid \left<\abf_j, \xbf \right> \geqslant nb_j , \left<\abf_l,\xbf \right> \leqslant nb_l \text{ for } j \leqslant p; \  p+1\le l\leqslant q\} \subseteq \Rset_+^r,
\end{equation}
and hence $\overline C_n$ is a polytope.

We next claim that $\overline C_1$ is full dimensional. Indeed, for any $\ybf  \in C_1$, by Formula (\ref{EDStabL2}) we can choose a real number $\varepsilon > 0$ such that for all real numbers $\varepsilon_1, \ldots,\varepsilon_r$ with $0 \leqslant \varepsilon_1,\ldots,\varepsilon_r\leqslant \varepsilon$, we have $\ybf +\varepsilon_1\ebf_1+\cdots+\varepsilon_r \ebf_r \in C_1$. This means that the parallelotope  $[y_1, y_1+\varepsilon] \times \cdots \times [y_r, y_r+\varepsilon]\subseteq C_1$ , and thus $\overline C_1$ is full dimensional in $\Rset^r$, as claimed.

 Since the polytope $\overline C_1$ is full dimensional, by the Decomposition Theorem for polyhedra (see \cite[Corollary 7.1.b]{S}), we can find $r+1$ vertices, say $\albf_0,\ldots,\albf_r$, of the polytope $\overline C_1$, which are affinely independent. Let $\albf =\frac{1}{r+1}(\albf_0+\cdots+\albf_r)$ be the barycenter of the $r$-simplex $[\albf_0, \albf_1, \ldots , \albf_r]\subseteq \overline{C_1}$.

For each $i\leqslant r$, set  $\lambda_i =  \lceil \alpha_i\rceil -\alpha_i \geqslant 0$, where $ \lceil \alpha_i \rceil$ is the least integer which is bigger than or equal $\alpha_i$. Then $\lambda_1+\cdots+\lambda_r < r$ and $\gabf :=  n\albf +\lambda_1\ebf_1+\cdots+\lambda_r \ebf_r\in \Nset^r$.  In order to show $C_n\cap \Nset^r \neq \emptyset $, it suffices to show that
$\gabf \in C_n.$

Since $\albf\in \overline {C_1}$,  by Formula (\ref{EDStabL3}), we have
$$\left<\abf_l,\gabf \right> = \left<\abf_l, n\albf \right> + \left<\abf_l,\sum_{i=1}^r\lambda_i\ebf_i\right>\geqslant  n\left<\abf_l,\albf \right> \geqslant nb_l,$$
for all $l = p+1,\ldots,q$.  

Now, fix an index $j \leqslant p$.  Since $\albf_0,\albf_1,\ldots,\albf_r$ are affinely independent in $\Rset^r$, there is at least one point not lying in the hyperplane
$\left<\abf_j,\xbf \right> = b_j$. We may assume that $\albf_0$ is such a point. From Formula (\ref{EDStabL3}) we then have
\begin{equation}\label{EDStabL4}
\left<\abf_j,\albf_0\right> < b_j \ \text{ and } \left<\abf_j,\albf_i\right> \leqslant b_j \text{ for all } i=1,\ldots, r.
\end{equation}
Since $\albf_0$ is a vertex of the polytope $\overline C_1$, by \cite[Formula 23 in Page 104]{S}, $\albf_0$ can be represented as the unique solution of a system of linear equations of the form:
$$\left<\abf_h,\xbf \right>  = b_h \ \text{ for  all } h \in S \subseteq \{1,\ldots,q\},$$
where $|S| = r$. By Cramer's rule we have
$ \alpha_{0i} = \delta_i/{\delta }$  for all $i\leqslant r$, where $\delta ,\ \delta_1, \ldots ,\ \delta_r \in \Nset$ and $\delta$ is the absolute value of the determinant of this system of linear equations. In particular, $\delta  \albf_0\in\Nset^r$.
Using the inequalities (\ref{EDStabL1}) and  Hadamard's inequality applied to $\delta$, we get 
\begin{equation}\label{EDStabL7}
\delta  \leqslant r^{r/2}(r-1)^rd(I)^{r(r-2)}.
\end{equation}
By  (\ref{EDStabL4}) we have $\left<\abf_j, \delta \albf_0\right> < \delta b_j$, whence $\left<\abf_j, \delta \albf_0\right> \leqslant \delta b_j-1$ because $\abf_j , \ \delta \albf_0\in\Nset^r$. Let $c = n/ (r+1)\delta $, then by (\ref{EDStabL7}),  $c\geqslant r(r-1)d(I)^{r-2}$. We then have
\begin{align*}
\left<\abf_j, n \albf \right>  &= c \delta \left<\abf_j, (r+1)\albf\right> \\
& = c\delta  \sum_{i=0}^{r}\left<\abf_j, \albf_i \right>\\
& = c \left<\abf_j, \delta \albf_0 \right>+ c\delta   \sum_{i=1}^{r}\left<\abf_j, \albf_i \right>  \\
&\leqslant c(\delta b_j-1)+ rc \delta b_j \\
& = nb_j - c.
\end{align*}
Hence
\begin{align*}
\left<\abf_j, \gabf  \right> &= \left<\abf_j, n \albf \right>+\left<\abf_j, \sum_{i=1}^r \lambda_i\ebf_i\right> \\
& = \left<\abf_j,  n \albf \right>+ \sum_{i=1}^r \lambda_i a_{ji}\\
& < nb_j - c + r (r-1)d(I)^{r-2}  \ \  \text{(by (\ref{EDStabL1}) and } \sum \lambda_i < r) \\
& \leqslant nb_j.
\end{align*}
So $\left<\abf_j, \gabf \right> \leqslant  nb_j-1$, for all $j \leqslant p$.  This means that $\gabf \in  C_n$, as required.
\end{proof}

We are now in position to prove the  main result of this section.

\begin{thm}\label{DStab}  Let $I$ be a monomial ideal of $R$.  Let
 \begin{equation*}n_1(I):=
\begin{cases}
1 & \text{ if } r \leqslant 2,\\
r(r^2 - 1)r^{r/2}(r-1)^r d(I)^{(r-2)(r+1)}  & \text{ if } r > 2.
\end{cases}
\end{equation*}
Then, $\depth R/\overline{I^n} = \dim R - \ell(I)$ for all $n \geqslant n_1(I)$.
\end{thm}

\begin{proof} We prove the theorem by induction on $r$. If $r \leqslant 2$, and $I \neq 0$, then $\depth R/\overline{I^n} = 0,1$, and $\depth R/\overline{I^n} = 0$ if and only if $\mfr \in \ass(R/\overline{I^n})$. Since $\ass(R/\overline{I^n})$ is constant for all $n\geqslant 1$ in this case (by \cite [Proposition 16]{ME}), we get
$$\depth R/\overline{I^n} = \depth R/\overline{I} \ \text{ for all } n\geqslant 1,$$
and the theorem follows from Lemma \ref{DepthInc}(2).

Assume that $r\geqslant 3$. By virtue of Lemma \ref{DepthInc}(2) and symmetry,  it suffices to show that  
\begin{equation}\label{EDStab1} \depth R/\overline{I^n} \leqslant \depth R/ \overline{I^m},
\end{equation}
for any $m, n \geqslant n_1(I)$.  Let $t := \depth R/ \overline{I^m}$. As $H^t_{\mfr}(R/\overline{I^m}) \ne 0$, by Lemma \ref{Takay}, there is $\bebf \in \Zset^r$ such that 
\begin{equation}\label{EDStab2}
\dim_k \Htil_{t -|CS_{\bebf}|-1}(\Delta_{\bebf}(\overline{I^m});k)  = H^t_{\mfr}(R/\overline{I^m})_{\bebf}\ne 0.
\end{equation}
In particular, $\Delta_{\bebf}(\overline{I^m}) \ne \emptyset$.  If   $CS_{\bebf} = \emptyset$, i.e., $\bebf\in\Nset^r$, then (\ref{EDStab1}) follows from Lemma \ref{DStabL}.  

We now assume that $CS_{\bebf} \ne \emptyset$. Without loss of generality, we may assume that
$CS_{\bebf} = \{s + 1, \ldots , r\}$ for some integer $0\leqslant s \leqslant r$. If $s = 0$, i.e.  $CS_{\bebf} = [r]$, then $\Delta_{\bebf}(\overline {I^m}) = \{\emptyset\}$. By Lemma \ref{Takay}, it follows that $H^0_{\mfr}(R/ \overline{I^m}) = 0$, which is equivalent to $\mfr \in \ass R/\overline{I^m}$. Since $r\geqslant 3$, $n\geqslant n_1(I) > (r-1)rd(I)^{r-2}$,  $\mfr \in\ass R/\overline{I^n}$ by Lemma \ref{AssL2}. Hence, $\depth R/\overline{I^n} = 0 = \depth R/\overline{I^m}$ in this case.

Assume that $s \geqslant 1$. Let $R' := k[x_1,\ldots,x_s] = R[\{s+1,...,r\}]$ (in the notation before Remark \ref{Restriction}) and $I':= I[\{s+1,...,r\}] \subseteq  R'$. Let $\bebf = (\beta_1,\ldots , \beta_r)$ and  $\bebf' := (\beta_1,\ldots , \beta_s) \in \Nset^s$. Then,  by Formula (\ref{EQ01}), $\Delta_{\bebf'}(\overline {I'^m}) = \Delta_{\bebf}(\overline{I^m})$. Let $\nfr  := (x_1,\ldots,x_s)$ be the maximal homogeneous ideal of $R'$. Using (\ref{EDStab2}) and  Lemma \ref{Takay} we obtain
\begin{align*}
\dim_k H_{\nfr}^{t - |CS_{\bebf}|} (R'/\overline{I'^m})_{\bebf'} &=\dim_k \Htil_{t-|CS_{\bebf}|-1}(\Delta_{\bebf'}(\overline{I'^m});k)\\
&=\dim_k \Htil_{t-|CS_{\bebf}|-1}(\Delta_{\bebf}(\overline{I^m});k) \ne 0.
\end{align*}
Hence $I' \neq R'$ and  $\depth R'/\overline{I'^m} \leqslant t -|CS_{\bebf}|$, or equivalently $\depth R/\overline{I^m} \geqslant |CS_{\bebf}| + \depth R'/\overline{I'^m}$.  On the other hand, by \cite[Lemma 1.3]{HKTT}, we have $\depth R/\overline{I^m} \leqslant |CS_{\bebf}| + \depth R'/\overline{I'^m}$ and $\depth R/\overline{I^n} \leqslant  |CS_{\bebf}| + \depth R'/\overline{I'^n}$. Hence
$$ \depth R/\overline{I^m} = |CS_{\bebf}| + \depth R'/\overline{I'^m},$$
and (noticing that $n_1(I') \leqslant n_1(I)$, since $d(I')\leqslant d(I)$ and $s\leqslant r$)
$$\begin{array}{ll}
\depth R/\overline{I^n} & \leqslant |CS_{\bebf}|+ \depth R'/\overline{I'^n} \\
& =  |CS_{\bebf}|+ \depth R'/\overline{I'^m}  \ \ \text{(by the induction hypothesis) } \\
& =\depth R/\overline{I^m}.
\end{array}$$
\end{proof}

\noindent {\it Remark}. Set 
$$ \overline{\dstab}(I) = \min\{m |\  \depth R/ \overline{I^n} = \depth R/\overline{I^m} \ \text{for all } n\geqslant m\}.$$
One can call it  the index of depth stability  for integral closures. Then Theorem \ref{DStab} says that  $\overline{\dstab}(I)\le n_1(I)$.  It seems that this bound is too big. However, an example given in \cite[Proposition 17]{Tr1} shows that an upper bound on $ \overline{\dstab}(I)$ must be at least of the order $d(I)^{r-2}$.

\section{Cohen-Macaulay property} \label{CMness}

In this section we  apply  results in previous sections to study the  Cohen-Macaulayness of integeral closures of powers of  monomial ideals. We say that $I$ is equimultiple if $\ell(I) = \height(I)$. Note that, by \cite[Theorem 2.3]{BA}, we can compute $\ell(I)$ in terms of geometry of  $NP(I)$.
$$\ell(I) = \max\{\dim F + 1 \mid F \text{ is a compact face of } NP(I)\}.$$
Therefore, the condition $I$ being equimultiple is independent on the characteristic of the base field $k$.

\begin{thm}\label{CM} Let $I$ be a monomial ideal of $R$. The following conditions are equivalent
\begin{enumerate}
\item $R/\overline{I^n}$ is a Cohen-Macaulay ring for all $n\geqslant 1$,
\item  $R/\overline{I^n}$ is a Cohen-Macaulay ring for some $n\geqslant n_1(I)$, where $n_1(I)$ is defined in Theorem \ref{DStab},
\item $I$ is an equimultiple ideal of $R$. 
\end{enumerate}
\end{thm}
\begin{proof} If $R/\overline{I^n}$ is a Cohen-Macaulay ring for some $n \ge  n_1(I)$, then by Theorem \ref{DStab} we have
$\dim R/\overline{I^n} = \depth R/\overline {I^n} =\dim R -\ell(I) .$
On the other hand, $\dim R/\overline{I^n} = \dim R/I = \dim R-\height(I)$.  Hence, $\ell(I) = \height(I)$.

Conversely, assume that $\ell(I) =\height(I)$. Then, 
$$\depth R/\overline I \leqslant \dim R/\overline I = \dim R/ I = \dim R-\height(I) = \dim R-\ell(I).$$
For all $n\geqslant 1$, by Lemma \ref{DepthInc} applied to $m\gg 0$, we have
$$\dim R -\ell(I) \geqslant \depth R/\overline I \geqslant \depth R/\overline {I^n} \geqslant \depth R/\overline {I^{mn}}  = \dim R-\ell(I).$$
Hence,
$$\depth R/\overline {I^n} = \dim R-\ell(I) = \dim R -\height(I) = \dim R/I = \dim R/\overline {I^n},$$
which means that $R/\overline{I^n}$ is a Cohen-Macaulay ring.
\end{proof}

In the rest of this section we will improve the above theorem for the class of square-free monomial ideals.  We need some auxiliary results.

\begin{lem}\label{PrimDec}  Let $I$ be an unmixed monomial ideal and $I = Q_1\cap\cdots\cap Q_s$ be an irredundant primary decomposition of $I$. Assume that $\overline{I^n}$ is unmixed for some $n\geqslant 1$. Then
$$\overline{I^n} = \overline{Q_1^n}\cap\cdots\cap \overline{Q_s^n}.$$
\end{lem}

\begin{proof} We  prove by induction on $s$ and on $r$.  If $s=1$ (which also includes the case $r=1$) there is nothing to prove. 
Assume that $s\geqslant 2$ and $r\geqslant 2$.  Since $I$ is unmixed, $\sqrt{Q_j} \neq \mfr$ for all $j\leqslant s$.   For each $i\leqslant r$, let
$$A_i := \{j \leqslant s \mid x_i \notin \sqrt{Q_j}\}.$$
If there is $j_0\le s$ such that $j_0 \not\in \cup_{i=1}^r A_i$, then $x_1,...,x_r \in \sqrt{Q_{j_0}}$, a contradiction. Hence $ \cup_{i=1}^r A_i = [s]$ and $I = \bigcap_{i=1}^r(\cap_{j\in A_i}Q_j))$, where we set $ \cap_{j\in A_i}Q_j = R$ if $A_i = \emptyset $. Moreover, using Remark \ref{Restriction}(1) and the induction hypothesis on $r$, we may assume that $|A_i| <s$ for all $i\leqslant r$.

It is well-known that one can get a primary decomposition of a monomial ideal $\afr \subset R$ by repeated application of the formula $(B, uv) = (B, u) \cap (B,v)$, where $B$ is a set of monomials and $u,v$ are monomials having no common variable. Based on this fact, it is immediate to see that one can get a primary decomposition of $\afr[i] R$ from that of  $\afr$ by deleting those primary components whose associated prime ideals contain $x_i$ (recall that $\afr[i]$ is obtained from $G(I)$ by setting $x_i = 1$).  

Using this remark  we see that  $I[i]$  is an unmixed ideal for all $i\leqslant r$. Moreover $I[i] = \cap_{j\in A_i} Q_j$  and $I = \cap_{i=1}^r I[i]$. 

By Remark \ref{Restriction}(1)  $\overline{I[i]^nR} = (\overline{I^n})[i] R$ and $\overline{I^n} = \cap_{i=1}^r \overline{I[i]^nR}$. Since $\overline{I^n}$ is unmixed, by the above remark, all  $\overline{I[i]^nR} $ are  unmixed ideals. Since $|A_i| <s$,  by the induction hypothesis on $s$, we also get $\overline{I[i]^nR} = \cap_{j\in A_i} \overline{Q_j^n}$. 

Let $J := \overline{Q_1^n}\cap\cdots\cap \overline{Q_s^n}$.  This is an unmixed ideal. Hence, as shown above,  $J =\bigcap_{i=1}^r J[i]R$ and $J[i]R = \cap_{j\in A_i} \overline{Q_j^n} = \overline{I[i]^nR}$ . Then $J = \cap_{i=1}^r \overline{I[i]^nR} = \overline{I^n}$, as required.
\end{proof}

\begin{lem}\label{Ext} Let $y$ be a new variable and $S = R[y] = k[x_1,...,x_r,y]$. Then, for every $n\geqslant 1$ we have
\begin{enumerate}
\item $\overline{(I,y)^n} = \sum_{i=0}^n y^i \overline{I^{n-i}}S$.
\item $\overline{(I,y)^n}$ is Cohen-Macaulay if and only if $\overline{I^i}$ is Cohen-Macaulay for all $i\leqslant n$.
\end{enumerate}
\end{lem}

\begin{proof} (1) The inclusion $\sum_{i=0}^n y^i \overline{I^{n-i}}S\subseteq \overline{(I,y)^n}$ follows from the fact that $\overline{I_1}\cdot \overline{I_2} \subseteq \overline{I_1I_2}$ for all ideals $I_1$ and $I_2$ in $S$.

In order to prove the reverse inclusion, let $G(I) = \{\xbf^{\albf_1}, \ldots, \xbf^{\albf_s}\}$.  Set $\albf^*_j = (\albf_j, 0) \in \Nset^{r+1}$ and let $\ebf_{r+1}$ be the $(r+1)$-th unit vector of $\Rset^{r+1}$. Assume that  $(\albf, \beta ) \in NP((I,y)^n) \cap  \Nset^{r+1}$.  From (\ref{EN2}) we see that there are non-negative numbers $a_1, ...,a_s, b $ such that $\sum_{j=1}^ra_j + b \geqslant n$ and 
$(\albf,\beta ) = \sum_{j=1}^sa_j \albf_j^* + b\ebf_{r+1} .$
Then $  b  = \beta \in \Nset $. If $b \geqslant n$, then $\xbf^{\albf}y^\beta  \in y^nS$. Assume that $b<n$. Then $\sum_{j=1}^r a_j \geqslant n-b>0$  and $\albf =  \sum_{j=1}^sa_j \albf_j  \in NP(I^{n-b})$. Hence, by (\ref{EN1}), $\xbf^\albf y^\beta \in y^m \overline{I^{n-m}}S$.  In both cases, $\xbf^\albf y^\beta \in  \sum_{i=0}^n y^i \overline{I^{n-i}}S$, i.e., $\overline{(I,y)^n} \subseteq  \sum_{i=0}^n y^i \overline{I^{n-i}}S$.

(2) From (1) we deduce that
$$S/\overline{(I,y)^n} \cong R/\overline{I} \oplus R/\overline{I^2}\oplus \cdots\oplus R/\overline{I^n}$$
as $R$-modules, and the conclusion follows.
\end{proof}

From now on, let $I$ be an ideal generated by square-free monomials. Such an ideal is often called a Stanley-Reisner ideal and is associated to the simplicial complex $\Delta := \Delta (I)$. In this case we also denote $I$ by $I_\Delta $.  Note that we do not require that $\Delta $ contains all vertices $\{i \}$, $i\leqslant r$.  Recall that for a face $F\in \Delta $, the link of $F$ is defined by 
$$\link_\Delta (F) = \{ G \subseteq [r]\setminus F|\ F\cup G \in \Delta \}.$$
We simply write $\link_\Delta i$ for $\link_\Delta \{i\}$.

\begin{cor}\label{CMlink} If $\overline{I_{\Delta}^n}$ is Cohen-Macaulay, then so is $\overline{I_{\link_\Delta i}^n}$ for every vertex $i$ of $\Delta$.
\end{cor}
\begin{proof} Let $S = k[x_j \mid j\ne i]$. Then  $R = S[x_i]$. Let $J = I_{\Delta}R[x_i^{-1}] \cap S$. We have $\overline{J^n} = \overline{I_{\Delta}^n}R[x_i^{-1}] \cap S$, whence $\overline{J^n}$ is Cohen-Macaulay.

Denote by $V(\Delta )$ the set of vertices of a simplicial complex $\Delta $. Let  $Y = \{x_j\mid j \notin V(\link_{\Delta} i) \text{ and } j \ne i\}$. By \cite[Lemma 2.1]{TT} we have $$I_{\Delta}R[x_i^{-1}] = (I_{\link_{\Delta} i}, Y) R[x_i^{-1}]. $$ 

It follows that $J = (I_{\link_{\Delta} i}, Y) $ as  ideals in $S$. By Lemma \ref{Ext} we conclude that $\overline{I_{\link_{\Delta} i}^n}$ is Cohen-Macaulay.
\end{proof}

Note that $I_{\Delta}$ is a complete intersection if and only if any two of its minimal monomial generators have no common variable. Recall that $\dim \Delta = \max\{|F|| \ F\in \Delta \} -1$. 

\begin{lem}\label{CMDim0} Assume that $\dim \Delta = 0$ and that $\overline{I_{\Delta}^n}$ is Cohen-Macaulay for some $n\geqslant 2$. Then, $I_\Delta$ is a complete intersection. Moreover, $\Delta$ has at most two vertices.
\end{lem}
\begin{proof} Since $\dim \Delta = 0$, we may assume that $\Delta= \left<\{1\},\ldots,\{s\}\right>$ for some $s\leqslant r$. Then 
$$I_{\Delta} = (x_{s+1},...,x_r) +  \cap_{i=1}^s(x_1,\ldots,\widehat{x_i}, \ldots,x_s)=(x_{s+1},...,x_r; x_ix_j\mid 1\leqslant i < j \leqslant s).$$
If $s\leqslant 2$, then $I_\Delta$ is a complete intersection. It remains to show that $s \leqslant 2$.
Assume on the contrary that $s\geqslant 3$. Since
$$\overline{I_{\Delta}^2} = \overline{(x_{s+1},...,x_r, x_ix_j\mid 1\leqslant i < j \leqslant s)^2},$$
we can check that $\mfr = \overline{I^2}\colon (x_1x_2x_3)$,  but $x_1x_2x_3 \not\in I_{\Delta}^2$. Hence $\mfr \in \ass R/\overline{I^2}$. Using the non-decreasing property of the sets associated prime ideals of integral closures of powers of an ideal (see  \cite[Proposition 16.3]{HIO}), we have $\mfr\in \ass R/\overline{I^n}$. Hence  $\overline{I_{\Delta}^n}$ is not Cohen-Macaulay, a contradiction. 
\end{proof}

\begin{lem}\label{CMDim1} Assume that $\dim \Delta = 1$ and that $\overline{I_{\Delta}^n}$ is Cohen-Macaulay for some $n\geqslant 3$. Then, $I_\Delta$ is a complete intersection. 
\end{lem}

\begin{proof} Since $\overline{I_{\Delta}^n}$ is Cohen-Macaulay and $\sqrt{\overline{I_{\Delta}^n}} = I_{\Delta}$,  $\Delta$ is Cohen-Macaulay by \cite[Theorem 2.6]{HTT}. Since $\dim \Delta = 1$, this property implies that $\Delta$ is connected. In particular, every facet of $\Delta$ has exactly two vertices, and we can regard $\Delta$ as a connected graph without isolated vertices. We may assume that $V(\Delta ) = [s]$ for some $s\leqslant r$.

If $s = 2$, then $\Delta$ is just an edge, and $I_{\Delta} = (x_3,...,x_r)$.  Assume that $s\geqslant 3$. For each vertex $i$ of $\Delta$, by Corollary \ref{CMlink} and Lemma \ref{CMDim0},  $\link_{\Delta}(i)$ is either one vertex or consists of exactly two vertices. Consequently, $\Delta$ is either a path or a cycle. 

For each edge $\{i,j\}$ of $\Delta$ , set $P_{ij} =(x_{s+1},...,x_r; x_l \mid 1\le l \leqslant s; \  l\ne i \text{ and } l \ne j)$. Then,
$I_{\Delta} = \cap_{\{i,j\} \in \Delta} P_{ij}.$
Since $\overline{I_{\Delta}^n}$ is unmixed, by Lemma \ref{PrimDec} we have
$$\overline{I_{\Delta}^n} = \cap_{\{i,j\} \in \Delta} \overline{P_{ij}^n} = \cap_{\{i,j\} \in \Delta} P_{ij}^n = I_{\Delta}^{(n)},$$
where $I_{\Delta}^{(n)}$ is the $n$-th symbolic power of $I_\Delta$. This implies that $I_{\Delta}^{(n)}$ is Cohen-Macaulay. By \cite[Theorem 2.4]{MT1},  every pair of disjoint edges of $\Delta$ is contained in a cycle of length $4$. Since $\Delta$ is either a path or a cycle, we conclude that $\Delta$ is either a path of length  two, or a cycle of length 3 or 4. Hence, either $I_{\Delta} = (x_4,...,x_r; x_1x_3), \   (x_4,...,x_r; x_1x_2x_3)$ or $I_{\Delta} =  (x_5,...,x_r; x_1x_3, x_2x_4)$ - all are complete intersections.
\end{proof}

We can now improve Theorem \ref{CM} for square-free monomial ideals by giving an exact description of all square-free monomial  ideals $I_{\Delta}$ such that  $\overline{I_{\Delta}^n}$ is Cohen-Macaulay for some $n\geqslant 3$.

\begin{thm}\label{CMDelta} Let $\Delta$ be a simplicial complex. Then the following conditions are equivalent:
\begin{enumerate}
\item $\overline{I_{\Delta}^n}$ is Cohen-Macaulay for every $n\geqslant 1$;
\item $\overline{I_{\Delta}^n}$ is Cohen-Macaulay for some $n\geqslant 3$;
\item $I_\Delta$ is a complete intersection;
\item $I_{\Delta}$ is an equimultipe ideal.
\end{enumerate}
\end{thm}

\begin{proof} $(1) \Rightarrow  (2)$  and  $(3) \Rightarrow   (4)$ are clear. $(4) \Rightarrow  (1)$ follows from Theorem \ref{CM}.

It remains to prove that $(2) \Rightarrow (3)$. The following proof is similar to that of \cite[Theorem 4.3]{TT}.  We prove the implication by induction on $\dim \Delta$. The case $\dim \Delta \leqslant 1$ follows from Lemmas \ref{CMDim0} and \ref{CMDim1}. Assume that $\dim \Delta \geqslant 2$. Since $\overline{I_{\Delta}^n}$ is Cohen-Macaulay and $\sqrt{\overline{I_{\Delta}^n}} = I_{\Delta}$,  $I_\Delta$ is Cohen-Macaulay by \cite[Theorem 2.6]{HTT}. In particular, $\Delta$ is connected.  On the other hand, 
 by Corollary \ref{CMlink},  $\overline{I_{\link_{\Delta} i}^n}$ is Cohen-Macaulay for all $i\leqslant s$, where w.l.o.g. we assume $V(\Delta ) = [s]$. By the induction hypothesis,  $I_{\link_{\Delta}} i$ is a complete intersection. 
Since $\Delta$ is connected, this implies that $I_\Delta$ is a complete intersection complex by \cite[Theorem 1.5]{TY}. 
\end{proof}

Note that \cite[Theorem 1.2]{TT} states that the last condition in Theorem \ref{CMDelta} is also equivalent to the Cohen-Macaulayness of $R/I^n$ for all $n\geqslant 1$ (or for some fixed $n\geqslant 3$).

The property that the Cohen-Macaulay property of $\overline{I^n}$ for some $n\geqslant 3$   forces that for all $n$ is very specific for  square-free monomial ideals. For an arbitrary monomial ideal, the picture is much more complicate, as shown by the following example.

\begin{exm} Let $d\geqslant 3$ and $I=(x^d,xy^{d-2}z,y^{d-1}z) \subset R=k[x,y,z]$. Then $(x,y,z) \in \ass(R/\overline{I^n})$ if and only if $n\geqslant d$ (see the example in  \cite[Page 54]{Tr1}). Since $\dim R/I = 1$, it follows that:
\begin{enumerate}
\item $R/\overline{I^n}$ is Cohen-Macaulay for each $n=1,\ldots,d-1$;
\item $R/\overline{I^n}$ is not Cohen-Macaulay for any $n\geqslant d$.
\end{enumerate}
Note that $\height(I) = 2$ and $\ell(I) = 3$ in this case.
\end{exm}

\subsection*{Acknowledgment} This work is partially supported by NAFOSTED (Vietnam) under the grant number 101.04-2015.02.

\end{document}